\documentclass[11pt]{amsart}%
\usepackage{graphicx,amscd,color,amsmath,amsfonts,amssymb,geometry}
\usepackage{amsmath}
\usepackage{amsfonts}
\usepackage{amssymb}
\usepackage{graphicx}%
\providecommand{\U}[1]{\protect\rule{.1in}{.1in}}
\newtheorem{theorem}{Theorem}[section]
\newtheorem{proposition}[theorem]{Proposition}
\newtheorem{corollary}[theorem]{Corollary}
\newtheorem{lemma}[theorem]{Lemma}

\theoremstyle{definition}

\newtheorem{problem}[theorem]{Problem}

\begin{document}
\title[Unified Grothendieck's and Kwapie\'{n}'s theorems for multilinear operators]{Unified Grothendieck's and Kwapie\'{n}'s theorems for multilinear operators}
\author[D. N\'u\~nez]{Daniel N\'{u}\~{n}ez-Alarc\'{o}n}
\address{Departamento de Matem\'{a}ticas\\
\indent Universidad Nacional de Colombia\\
\indent111321 - Bogot\'a, Colombia}
\email{danielnunezal@gmail.com and dnuneza@unal.edu.co}
\author[J. Santos]{Joedson Santos}
\address{Departamento de Matem\'{a}tica \\
Universidade Federal da Para\'{\i}ba \\
58.051-900 - Jo\~{a}o Pessoa, Brazil.}
\email{joedson.santos@academico.ufpb.br}
\author[D. Serrano]{Diana Serrano-Rodr\'{\i}guez}
\address{Departamento de Matem\'{a}ticas\\
\indent Universidad Nacional de Colombia\\
\indent111321 - Bogot\'{a}, Colombia}
\email{dmserrano0@gmail.com and diserranor@unal.edu.co}
\thanks{J. Santos is supported by CNPq and Grant 2019/0014 Para\'{\i}ba State Research
Foundation (FAPESQ)}
\subjclass[2010]{}

\begin{abstract}
Kwapie\'{n}'s theorem asserts that every continuous linear operator from
$\ell_{1}$ to $\ell_{p}$ is absolutely $\left(  r,1\right)  $-summing for
$1/r=1-\left\vert 1/p-1/2\right\vert .$ When $p=2$ it recovers the famous
Grothendieck's theorem. In this paper investigate multilinear variants of
these theorems and related issues. Among other results we present a unified
version of Kwapie\'{n}'s and Grothendieck's results that encompasses the cases
of multiple summing and absolutely summing multilinear operators.

\end{abstract}
\maketitle

\section{Introduction}

Let $E,F$ be Banach spaces and $r\geq s\geq1$ be real numbers. A continuous
linear operator $T:E\rightarrow F$ is absolutely $\left(  r,s\right)
$-summing if $\left(  T(x_{j})\right)  _{j=1}^{\infty}\in\ell_{r}(F)$ whenever
$\left(  x_{j}\right)  _{j=1}^{\infty}\in\ell_{s}^{w}(E),$ where $\ell_{s}%
^{w}(E)$ denotes the space of weakly $s$-summable sequences in $E$, i.e., the
sequences $(x_{j})_{j=1}^{\infty}$ in $E$ such that%
\[
\left\Vert (x_{j})_{j=1}^{\infty}\right\Vert _{w,p}:=\sup_{\varphi\in
B_{E^{\ast}}}\left(  \sum\limits_{j=1}^{\infty}\left\vert \varphi
(x_{j})\right\vert ^{p}\right)  ^{1/p}<\infty.
\]
One of the cornerstones of the theory of absolutely summing operators is
Grothendieck's theorem, which asserts that every continuous linear operator
from $\ell_{1}$ to $\ell_{2}$ is absolutely $\left(  1,1\right)  $-summing. In
\cite{K}, Kwapie\'{n} extended Grothendieck's theorem replacing $\ell_{2}$ by
$\ell_{p}$ as follows: every continuous linear operator from $\ell_{1}$ to
$\ell_{p}$ is absolutely $\left(  r,1\right)  $-summing, with%
\begin{equation}
1/r=1-\left\vert 1/p-1/2\right\vert , \label{rrr}%
\end{equation}
and this result is optimal (see also \cite{bennett}). In the last decades the
notion of absolutely summing operators was extended to the multilinear and
nonlinear setting in several different lines of research (see \cite{blocos,
bay111, bo1, defant} and \cite{diestel} for the linear theory). In this paper
we shall be interested in the notions of absolutely summing multilinear
operators and multiple summing multilinear operators (for the precise
definitions, see Section 2).

The extension of Kwapie\'{n}'s theorem to multilinear operators is a natural
problem to be investigated. For multiple summing operators, an immediate
consequence of \cite[Corollary 4.3]{nachr} is that every continuous $m$-linear
operator from $\ell_{1}$ to $\ell_{p}$ is multiple $\left(  r,1\right)
$-summing, with $r$ as in (\ref{rrr}) and this result is sharp. For absolutely
summing multilinear operators, as proved in \cite{bayart}, every continuous
$m$-linear operator from $\ell_{1}$ to $\ell_{p}$ is absolutely $\left(
r,1\right)  $-summing for%
\begin{equation}
r=\left\{
\begin{array}
[c]{c}%
\frac{2p}{mp+2p-2},\text{ if }1\leq p\leq2\\
\frac{2p}{mp+2},\text{ if }2\leq p\leq\infty.
\end{array}
\right.  \label{sss}%
\end{equation}
However, the optimality of the estimates (\ref{sss}) is not proven. Our first
result shows that for $2\leq p\leq\infty$ the above estimate is sharp.

The following variant of Kwapie\'{n}'s theorem was proved in \cite{defant}:

\begin{theorem}
[see \cite{defant}]\label{dd11}Let $T\in\mathcal{L}\left(  ^{m}\ell_{1}%
;\ell_{p}\right)  $ and $A_{k}\in\mathcal{L}\left(  ^{n}\ell_{\infty};\ell
_{1}\right)  $ for all $k=1,...,m.$ The composition $T\left(  A_{1}%
,...,A_{m}\right)  $ is multiple $\left(  r,1\right)  $-summing for%
\[
r=\left\{
\begin{array}
[c]{c}%
\frac{2n}{n+2-\frac{2}{p}},\text{ if }1\leq p\leq2\\
\frac{2n}{p+1},\text{ if }2\leq p\leq\frac{2n}{n-1}\\
2,\text{ if }\frac{2n}{n-1}\leq p\leq\infty.
\end{array}
\right.
\]

\end{theorem}

This result was recently improved in \cite{bayart} when $2\leq p\leq\infty$
and the authors also investigated the case of absolutely summing multilinear operators:

\begin{theorem}
[see \cite{bayart}]\label{dd22}Let $T\in\mathcal{L}\left(  ^{m}\ell_{1}%
;\ell_{p}\right)  $ and $A_{k}\in\mathcal{L}\left(  ^{n}\ell_{\infty};\ell
_{1}\right)  $ for all $k=1,...,m.$

(A) The composition $T\left(  A_{1},...,A_{m}\right)  $ is multiple $\left(
r,1\right)  $-summing for%
\[
r=\left\{
\begin{array}
[c]{c}%
\frac{2n}{n+2-\frac{2}{p}},\text{ if }1\leq p\leq2\\
\frac{2n}{n+\frac{2}{p}},\text{ if }2\leq p\leq\infty.
\end{array}
\right.
\]

(B) Assume that $n\geq2.$ The composition $T\left(  A_{1},...,A_{m}\right)  $
is absolutely $\left(  r,2\right)  $-summing for%
\[
r=\left\{
\begin{array}
[c]{c}%
\frac{2p}{mp+2p-2},\text{ if }1\leq p\leq2\\
\frac{2p}{mp+2},\text{ if }2\leq p\leq\infty.
\end{array}
\right.
\]

\end{theorem}

Note that while (A) provides $r$ so that $T\left(  A_{1},...,A_{m}\right)  $
is multiple $\left(  r,1\right)  $-summing, note that (B) provides $r$ so that
$T\left(  A_{1},...,A_{m}\right)  $ is absolutely $\left(  r,2\right)
$-summing. However, since every continuous $m$-linear operator $T\in
\mathcal{L}\left(  ^{m}\ell_{\infty};\ell_{1}\right)  $ is multiple $\left(
s;s\right)  $-summing for every $s\geq2$ (see \cite[Corollary 4.10]{complu})
it is obvious that the composition $T\left(  A_{1},...,A_{m}\right)  $ is
multiple $\left(  s,s\right)  $-summing for all $s\geq2$ and this result is
optimal in the sense that one cannot improve $\left(  s,s\right)  $ to
$\left(  r,s\right)  $ for $r<s.$ Thus, in the context of multiple summing
operators the nontrivial problem seems to be:

\begin{problem}
\label{p112233}Given $1\leq s<2,$ and positive integers $m,n,$ what is the
best $r$ so that the composition $T\left(  A_{1},...,A_{m}\right)  $ is
multiple $\left(  r,s\right)  $-summing for every $T\in\mathcal{L}\left(
^{m}\ell_{1};\ell_{p}\right)  $ and $A_{k}\in\mathcal{L}\left(  ^{n}%
\ell_{\infty};\ell_{1}\right)  $ for $k=1,...,m$?
\end{problem}

The paper is organized as follows. In Section 2 we present some preliminary
concepts and results which shall be used throughout the paper. In Section 3 we
prove that the estimate (\ref{sss}) provided by Bayart, Pellegrino and Rueda
in \cite{bayart} is optimal when $2\leq p\leq\infty$. In Section 4 we prove
Kwapie\'{n}'s and Grothendieck's inequalities for blocks and, finally, in
Section 5 we prove variants of Theorems \ref{dd11} and \ref{dd22}, providing a
partial answer to Problem \ref{p112233}.

\section{Background and notation}

Henceforth $\mathbb{K}$ represents the field of all scalars (complex or real),
$E,E_{1},...,E_{m},F$ denote Banach spaces over $\mathbb{K}$ and the Banach
space of all bounded $m$-linear operators from $E_{1}\times\cdots\times E_{m}$
to $F$ is denoted by $\mathcal{L}(E_{1},...,E_{m};F)$ and we endow it with the
classical sup norm (when $E_{1}=\cdots=E_{m}=E$ we write $\mathcal{L}%
(^{m}E;F)$ instead of $\mathcal{L}(E_{1},\cdots,E_{m};F)$). The topological
dual of $E$ is denoted by $E^{\ast}$ and its closed unit ball is denoted by
$B_{E^{\ast}}.$ Throughout the paper, for $p\in\lbrack1,\infty]$, the symbol
$p^{\ast}$ denotes the conjugate of $p$, that is $1/p+1/p^{\ast}=1$ and, as
usual, $1^{\ast}=\infty$ and $\infty^{\ast}=1.$

For the sake of completeness, we shall recall the notions of absolutely
summing multilinear operators and multiple summing multilinear operators.

If $\left(  r,s\right)  \in(0,\infty)\times\lbrack1,\infty]$ and $1/r\leq
m/s,$ an operator $T\in\mathcal{L}(E_{1},...,E_{m};F)$ is absolutely $(r;s)$%
-summing is there is a constant $C>0$ be such that
\[
\left(  \sum\limits_{j=1}^{\infty}\left\Vert T(x_{j}^{(1)},...,x_{j}%
^{(m)})\right\Vert _{F}^{r}\right)  ^{\frac{1}{r}}\leq C\prod_{k=1}%
^{m}\left\Vert (x_{j_{k}}^{(k)})_{j_{k}=1}^{\infty}\right\Vert _{w,s}%
\]
for every $\left(  x_{j}^{(k)}\right)  _{j=1}^{\infty}\in\ell_{s}^{w}(E_{k}).$

The class of absolutely $(r;s)$-summing operators is denoted by $\mathcal{L}%
_{as,(r;s)}(E_{1},...,E_{m};F)$ and the infimum taken over all possible
constants $C>0$ satisfying the previous inequality defines a norm in
$\mathcal{L}_{as,(r;s)}(E_{1},...,E_{m};F)$, which is denoted by
$\pi_{as\left(  r;s\right)  }^{m}$ (see \cite{AlencarMatos}).

If $1\leq s\leq r<\infty,$ an operator $T:E_{1}\times\cdots\times
E_{m}\rightarrow F$ is multiple\emph{ }$(r;s)$-summing if there is a constant
$C>0$ be such that
\[
\left(  \sum\limits_{j_{1},...,j_{m}=1}^{\infty}\left\Vert T(x_{j_{1}}%
^{(1)},...,x_{j_{m}}^{(m)})\right\Vert _{F}^{r}\right)  ^{\frac{1}{r}}\leq
C\prod_{k=1}^{m}\left\Vert (x_{j_{k}}^{(k)})_{j_{k}=1}^{\infty}\right\Vert
_{w,s}%
\]
for every $\left(  x_{j_{k}}^{(k)}\right)  _{j_{k}=1}^{\infty}\in\ell_{s}%
^{w}(E_{k}).$

The class of multiple $(r;s)$-summing operators is denoted by $\Pi
_{mult(r,s)}(E_{1},...,E_{m};F)$ (see \cite{collec, jmaa}). We recall the following inclusion theorem
(see \cite[Proposition 3.4]{PSDianaE} and \cite{nacib, bayart2} for extended
versions) that will be useful later:

\begin{theorem}
[see \cite{PSDianaE}]\label{9990}Let $m$ be a positive integer and $1\leq
s\leq u<\frac{mrs}{mr-s}.$ Then, for any Banach spaces $E_{1},...,E_{m},F$ we
have
\[
\Pi_{mult(r;s)}\left(  E_{1},\dots,E_{m};F\right)  \subset\Pi_{mult(\frac
{rsu}{su+mrs-mru};u)}\left(  E_{1},\dots,E_{m};F\right)
\]
and the inclusion has norm $1$.
\end{theorem}

\section{Optimality of Kwapie\'{n}'s inequality for multilinear operators}

In this section we show that (\ref{sss}) is optimal for $2\leq p\leq\infty.$
Let $0<r<\frac{2p}{mp+2}$. The proof is an adaptation or an argument used in
the proof of \cite[Theorem 1.1]{ps}. Let $n\in\mathbb{N}$ and $x_{1}%
,...,x_{n}\in\ell_{1}$ be non null vectors. Consider $x_{1}^{\ast}%
,...,x_{n}^{\ast}\in B_{\ell_{\infty}}$ so that $x_{j}^{\ast}(x_{j})=\Vert
x_{j}\Vert$ for every $j=1,...,n$. Let $a_{1},...,a_{n}$ be scalars such that
$\sum_{j=1}^{n}|a_{j}|^{p/r}=1$ and define the following $m$-linear operator%

\[
T_{n}:\ell_{1}\times\cdots\times\ell_{1}\longrightarrow\ell_{p},\ \ T_{n}%
(x^{(1)},...,x^{(m)})=%
{\textstyle\sum_{j=1}^{n}}
|a_{j}|^{\frac{1}{r}}x_{j}^{\ast}(x^{(1)})\cdots x_{j}^{\ast}(x^{(m)})e_{j}%
\]
where $e_{j}$ is the $j$-th canonical vector of $\ell_{p}$. Note that, for
every $(x^{(1)},...,x^{(m)})\in\ell_{1}\times\cdots\times\ell_{1}$, we have%
\begin{align*}
\Vert T_{n}(x^{(1)},...,x^{(m)})\Vert &  =\left(
{\textstyle\sum_{j=1}^{n}}
\left\vert |a_{j}|^{\frac{1}{r}}x_{j}^{\ast}(x^{(1)})\cdots x_{j}^{\ast
}(x^{(m)})\right\vert ^{p}\right)  ^{\frac{1}{p}}\\
&  \leq\left(
{\textstyle\sum_{j=1}^{n}}
|a_{j}|^{\frac{p}{r}}\right)  ^{\frac{1}{p}}\Vert x^{(1)}\Vert\cdots\Vert
x^{(m)}\Vert\\
&  =\Vert x^{(1)}\Vert\cdots\Vert x^{(m)}\Vert.
\end{align*}

It is plain that $T_{n}$ is absolutely $(r;1)$-summing. Note that for
$k=1,\ldots,n$, we have
\[
\Vert T_{n}(x_{k},....,x_{k})\Vert=\left\Vert \sum_{j=1}^{n}|a_{j}|^{\frac
{1}{r}}x_{j}^{\ast}(x_{k})^{m}e_{j}\right\Vert \geq|a_{k}|^{\frac{1}{r}}%
x_{k}^{\ast}(x_{k})^{m}=|a_{k}|^{\frac{1}{r}}\Vert x_{k}\Vert^{m}.
\]
Hence
\begin{align*}
\left(
{\textstyle\sum_{j=1}^{n}}
\Vert x_{j}\Vert^{mr}|a_{j}|\right)  ^{\frac{1}{r}}  &  =\left(
{\textstyle\sum_{j=1}^{n}}
\left(  \Vert x_{j}\Vert^{m}|a_{j}|^{\frac{1}{r}}\right)  ^{r}\right)
^{\frac{1}{r}}\\
&  \leq\left(
{\textstyle\sum_{j=1}^{n}}
\Vert T_{n}(x_{j},...,x_{j})\Vert^{r}\right)  ^{\frac{1}{r}}\\
&  \leq\pi_{r,1}(T_{n})\Vert(x_{j})_{j=1}^{n}\Vert_{w,1}^{m}.
\end{align*}
Since this last inequality holds whenever $\sum_{j=1}^{n}|a_{j}|^{\frac{p}{r}%
}=1$, if $\left(  p/r\right)  ^{\ast}$ is the conjugate index to $p/r,$ we
obtain
\begin{align*}
\left(  \sum_{j=1}^{n}\Vert x_{j}\Vert^{mr\left(  p/r\right)  ^{\ast}}\right)
^{\frac{1}{r\left(  p/r\right)  ^{\ast}}}  &  \leq\sup\left\{  \sum_{j=1}%
^{n}|a_{j}|\Vert x_{j}\Vert^{mr};\sum_{j=1}^{n}|a_{j}|^{\frac{p}{r}}=1\right\}
\\
&  \leq\left(  \pi_{r,1}^{m}(T_{n})\Vert(x_{j})_{j=1}^{n}\Vert_{w,1}%
^{m}\right)  ^{r}%
\end{align*}
and, then,
\begin{equation}
\frac{\left(
{\textstyle\sum_{j=1}^{n}}
\Vert x_{j}\Vert^{mr\left(  \frac{p}{r}\right)  ^{\ast}}\right)  ^{\frac
{1}{r\left(  \frac{p}{r}\right)  ^{\ast}}}}{\Vert(x_{j})_{j=1}^{n}\Vert
_{w,1}^{m}}\leq\pi_{r,1}^{m}(T_{n}). \label{01}%
\end{equation}
Since $0<r<\frac{2p}{mp+2}$ we have $mr\left(  p/r\right)  ^{\ast}<2$ and by
the Dvoretzky--Rogers Theorem (see \cite[Theorem 10.5]{diestel}), we know that
$id_{\ell_{1}}$ is not $\left(  mr\left(  p/r\right)  ^{\ast};1\right)
$-summing. Hence
\begin{equation}
\lim_{n\rightarrow\infty}\pi_{r,1}(T_{n})=\infty~\text{and}~\Vert T_{n}%
\Vert=1 \label{02}%
\end{equation}
and we conclude that the space of all absolutely $(r;1)$-summing $m$-linear
operators from $\ell_{1}$ to $\ell_{p}$ is not closed in $\mathcal{L}(^{m}%
\ell_{1};\ell_{p})$.

\section{Kwapie\'{n}'s theorem for blocks of sequences}

We shall need to introduce some terminology on tensor products. The product
\[
\widehat{\otimes}_{j\in\{1,\ldots,n\}}^{\pi}E_{j}=E_{1}\widehat{\otimes}^{\pi
}\cdots\widehat{\otimes}^{\pi}E_{n}%
\]
denotes the completed projective $n$-fold tensor product of $E_{1}%
,\ldots,E_{n}$. The tensor $x_{1}\otimes\cdots\otimes x_{n}$ is denoted for
short by $\otimes_{j\in\{1,\ldots,n\}}x_{j}$, whereas $\otimes_{n}x$ denotes
the tensor $x\otimes\cdots\otimes x$. In a similar way, $\times_{j\in
\{1,\ldots,n\}}E_{j}$ denotes the product space $E_{1}\times\cdots\times
E_{n}$. Let $n$ be a positive integer and $1\leq p_{1},\ldots,p_{n}<\infty$.
From now on, in this section, $r$ is defined by
\[
1/r=\min\left\{  1,%
{\textstyle\sum\limits_{i=1}^{n}}
1/p_{i}\right\}  .
\]
Let $D_{r}\subset\ell_{p_{1}}\widehat{\otimes}^{\pi}\cdots\widehat{\otimes
}^{\pi}\ell_{p_{n}}$ be the linear span of the tensors $\otimes_{n}e_{i}$ and
$\overline{D}_{r}$ be its closure.

The following result holds for $1\leq p_{1},\ldots,p_{n}\leq\infty,$ with
$\sum_{i=1}^{n}1/p_{i}<1$ (see \cite[Lemma 2.1]{blocos}):

\begin{lemma}
[{\cite[Lemma 2.1]{blocos}}]The map $u_{r}:\ell_{r}\rightarrow\overline{D_{r}%
}$, given by%
\[
u_{r}\left(
{\textstyle\sum\limits_{i=1}^{\infty}}
a_{i}e_{i}\right)  =%
{\textstyle\sum\limits_{i=1}^{\infty}}
a_{i}\,\otimes_{n}e_{i}%
\]
is an isometric isomorphism onto.
\end{lemma}

Following the ideas from \cite[Example 2.23(b)]{Ry}, the above result can be
easily complemented, now without the restriction $\sum_{i=1}^{n}1/p_{i}<1.$

\begin{lemma}
\label{comp}Let $n$ be a positive integer and $1\leq p_{1},\ldots,p_{n}%
\leq\infty$. The map $u_{r}:\ell_{r}\rightarrow\overline{D_{r}}$, given by
\[
u_{r}\left(
{\textstyle\sum\limits_{i=1}^{\infty}}
a_{i}e_{i}\right)  =%
{\textstyle\sum\limits_{i=1}^{\infty}}
a_{i}\,\otimes_{n}e_{i}%
\]
is an isometric isomorphism onto.
\end{lemma}

We also need the following results (see \cite{blocos}):

\begin{proposition}
[{\cite[Proposition 2.3]{blocos}}]\label{lin} Let $m$ be a positive integer
and let $E_{1},\dots,E_{m},F$ be Banach spaces. Let $1\leq k\leq m$ and
$I_{1},\ldots,I_{k}$ be pairwise disjoint non-void subsets of $\{1,\ldots,m\}$
such that $\cup_{j=1}^{k}I_{j}=\{1,\ldots,m\}$. Then given $T\in{\mathcal{L}%
}(E_{1},\ldots,E_{m};F)$, there is a unique $\widehat{T}\in{\mathcal{L}%
}(\widehat{\otimes}_{j\in I_{1}}^{\pi}E_{j},\ldots,\widehat{\otimes}_{j\in
I_{k}}^{\pi}E_{j};F)$ such that
\[
\widehat{T}(\otimes_{j\in I_{1}}x_{j},\dots,\otimes_{j\in I_{k}}x_{j}%
)=T(x_{1},\dots,x_{m})
\]
and $\Vert\widehat{T}\Vert=\Vert T\Vert$. The correspondence $T\leftrightarrow
\widehat{T}$ determines an isometric isomorphism between the spaces
${\mathcal{L}}(E_{1},\ldots,E_{m};F)$ and ${\mathcal{L}}(\widehat{\otimes
}_{j\in I_{1}}^{\pi}E_{j},\ldots,\widehat{\otimes}_{j\in I_{k}}^{\pi}E_{j};F)$.
\end{proposition}

Let $m$ be a positive integer and let $1\leq k\leq m$ and $\mathcal{I}%
=\left\{  I_{1},\ldots,I_{k}\right\}  $ be a family of pairwise disjoint
non-void subsets of $\{1,\ldots,m\}$ such that $\cup_{j=1}^{k}I_{j}%
=\{1,\ldots,m\}$. Let $\Omega_{I_{j}}\subset\mathbb{N}^{\left\vert
I_{j}\right\vert }$ be defined by%
\[
\Omega_{I_{j}}=Diag\left(  \mathbb{N}^{\left\vert I_{j}\right\vert }\right)
=\left\{  \left(  i,\ldots,i\right)  \in\mathbb{N}^{\left\vert I_{j}%
\right\vert }\right\}
\]
for all $j\in\{1,...,k\}$. Let us also define $\Omega_{\mathcal{I}}%
\subset\mathbb{N}^{m}$ as the product of diagonals%
\[
\Omega_{\mathcal{I}}=\Omega_{I_{1}}\times\cdots\times\Omega_{I_{k}}.
\]

We will denote by%
\[
\mathcal{L}_{\mathcal{I},as\left(  r;1\right)  }\left(  ^{m}\ell_{1};F\right)
\]
the space of all $T\in\mathcal{L}\left(  ^{m}\ell_{1};F\right)  $ such that%
\[%
{\textstyle\sum\limits_{\left(  j_{1},...,j_{m}\right)  \in\Omega
_{\mathcal{I}}}}
\left\Vert T\left(  x_{j_{1}}^{(1)},...,x_{j_{m}}^{(m)}\right)  \right\Vert
_{F}^{r}<\infty
\]
whenever
\[
\left(  x_{j_{i}}^{(i)}\right)  _{j_{i}=1}^{\infty}\in\ell_{1}^{w}\left(
\ell_{1}\right)  .
\]
Obviously
\[
\mathcal{L}_{\left\{  \left\{  1\right\}  ,\ldots,\left\{  m\right\}
\right\}  ,as\left(  r;1\right)  }\left(  ^{m}\ell_{1};F\right)
=\Pi_{mult(r;1)}\left(  ^{m}\ell_{1};F\right)
\]
and
\[
\mathcal{L}_{\left\{  1,\ldots,m\right\}  ,as\left(  r;1\right)  }\left(
^{m}\ell_{1};F\right)  =\mathcal{L}_{as\left(  r;1\right)  }\left(  ^{m}%
\ell_{1};F\right)  .
\]

The following lemma plays a crucial role to in the proof of the main theorem of
this section\textit{:}

\begin{lemma}
\label{lll}Let $F$ be a Banach space and let $\mathcal{I}=\left\{
I_{1},\ldots,I_{k}\right\}  $ be a family of pairwise disjoint non-void
subsets of $\{1,\ldots,m\}$ such that $\cup_{j=1}^{k}I_{j}=\{1,\ldots,m\}$. If
$n=\min\left\{  \left\vert I_{1}\right\vert ,\ldots,\left\vert I_{k}%
\right\vert \right\}  $, then the following assertions are equivalent:

(a) $\mathcal{L}\left(  ^{n}\ell_{1};F\right)  =\mathcal{L}_{as\left(
r;1\right)  }\left(  ^{n}\ell_{1};F\right)  $.

(b) $\mathcal{L}\left(  ^{m}\ell_{1};F\right)  =\mathcal{L}_{\mathcal{I}%
,as\left(  r;1\right)  }\left(  ^{m}\ell_{1};F\right)  $.
\end{lemma}

\begin{proof}
(a) implies (b). Let $A\in\mathcal{L}(^{m}\ell_{1};F)$ and fix in each set
$I_{i}$ an order. We define the sets $J_{1},...,J_{n}$ in the following way:
for every $j\in\left\{  1,...,n\right\}  ,$%
\[
\left(  J_{j}\right)  _{i}=\left(  I_{i}\right)  _{j},\text{ for }i\leq k,
\]
where $\left(  J_{j}\right)  _{i}$ means the $i-$th element of $J_{j}$ (in the
same way for $\left(  I_{i}\right)  _{j}$). The rest of elements that are not
yet in any $J_{j}$ are included in $J_{1}.$ Thus, $J_{1}$ has $m-\left(
n-1\right)  k$ elements and $J_{j}$, for $j\in\left\{  2,...,n\right\}  ,$ has
$k$ elements.

Obviously $\left\{  J_{1},...,J_{n}\right\}  $ is a family of non--void
pairwise disjoint subsets of $\{1,\ldots,m\}$ such that $\cup_{j=1}^{n}%
J_{j}=\{1,\ldots,m\}$. By Proposition \ref{lin}, let $\widehat{A}%
\in{\mathcal{L}}(\widehat{\otimes}_{j\in J_{1}}^{\pi}\ell_{1},\ldots
,\widehat{\otimes}_{j\in J_{n}}^{\pi}\ell_{1};F)$ be such that
\[
\widehat{A}(\otimes_{j\in J_{1}}x^{\left(  j\right)  },\dots,\otimes_{j\in
J_{n}}x^{\left(  j\right)  })=A(x^{\left(  1\right)  },\dots,x^{\left(
m\right)  })
\]
for every $x^{\left(  j\right)  }\in\ell_{1}$. Since $\widehat{\otimes}_{j\in
J_{j}}^{\pi}\ell_{1}$ is isometrically isomorphic to $\ell_{1}$, for all
$j\in\left\{  1,...,n\right\}  $, and by assumption we have that $\widehat
{A}\in\mathcal{L}_{as\left(  r;1\right)  }\left(  ^{n}\ell_{1};F\right)  $.
Then $\pi_{as\left(  r;1\right)  }^{n}$ $(\widehat{A})\leq M\Vert\widehat
{A}\Vert=M\Vert A\Vert$, where $M$ is a constant independent of $A$. We get
\begin{align*}
&  \left(
{\textstyle\sum\limits_{\left(  j_{1},...,j_{m}\right)  \in\Omega
_{\mathcal{I}\left(  m,n\right)  }}}
\left\Vert A(x_{j_{1}}^{\left(  1\right)  },\dots,x_{j_{m}}^{\left(  m\right)
})\right\Vert _{F}^{r}\right)  ^{\frac{1}{r}}\\
&  =\left(
{\textstyle\sum\limits_{\left(  j_{1},...,j_{m}\right)  \in\Omega
_{\mathcal{I}\left(  m,n\right)  }}}
\left\Vert \widehat{A}(\otimes_{i_{1}\in J_{1}}x_{j_{i_{1}}}^{\left(
i_{1}\right)  },\dots,\otimes_{i_{n}\in J_{n}}x_{j_{i_{n}}}^{\left(
i_{n}\right)  })\right\Vert _{F}^{r}\right)  ^{\frac{1}{r}}\\
&  \leq\pi_{as\left(  r;1\right)  }^{n}(\widehat{A})\left\Vert \left(
\otimes_{i_{1}\in J_{1}}x_{j_{i_{1}}}^{\left(  i_{1}\right)  }\right)
_{\left(  J_{1}\right)  _{1},...,\left(  J_{1}\right)  _{m-\left(  k-1\right)
n}=1}^{\infty}\right\Vert _{w,1}\prod_{s=2}^{n}\left\Vert \left(
\otimes_{i_{s}\in J_{s}}x_{j_{i_{s}}}^{\left(  i_{s}\right)  }\right)
_{\left(  J_{s}\right)  _{1},...,\left(  J_{s}\right)  _{k}=1}^{\infty
}\right\Vert _{w,1}\\
&  \leq M\Vert A\Vert K_{G}^{2\left(  m-\left(  k-1\right)  n\right)  -2}%
\prod_{i_{1}\in J_{1}}\left\Vert \left(  x_{j}^{\left(  i_{1}\right)
}\right)  _{j=1}^{\infty}\right\Vert _{w,1}\prod_{s=2}^{n}K_{G}^{2k-2}%
\prod_{i_{s}\in J_{s}}\left\Vert \left(  x_{j}^{\left(  i_{s}\right)
}\right)  _{j=1}^{\infty}\right\Vert _{w,1},\hspace*{25em}%
\end{align*}
where $K_{G}$ stands for Grothendieck's constant (see \cite[pg. 1420]{nachr}).
We thus conclude that $A\in\mathcal{L}_{\mathcal{I},as\left(  r;1\right)
}\left(  ^{m}\ell_{1};F\right)  $.

\medskip

(b) implies (a). Let $A:\ell_{1}\times\cdots\times\ell_{1}\rightarrow F$ be a
bounded $n$-linear operator. For each $s=1,\ldots,n$, by Lemma \ref{comp}, the
diagonal space $\overline{D}_{1}$ is complemented in $\widehat{\otimes}_{i\in
J_{s}}^{\pi}\ell_{1}$, and consider the diagonal projection $d_{1}$ from
$\widehat{\otimes}_{i\in J_{s}}^{\pi}\ell_{1}$ onto $\overline{D}_{1}$, such
that $d_{1}(\sum_{j_{1},...,j_{\left\vert J_{s}\right\vert }}a_{(j_{1}%
,...,j_{\left\vert J_{s}\right\vert })}e_{j_{1}}\otimes\cdots\otimes
e_{j_{\left\vert J_{s}\right\vert }})$ is equal to $\sum_{j_{1}%
,...,j_{\left\vert J_{s}\right\vert }}a_{(j_{1},...,j_{\left\vert
J_{s}\right\vert })}e_{j_{1}}\otimes\cdots\otimes e_{j_{\left\vert
J_{s}\right\vert }}$ if $j_{1}=\cdots=j_{\left\vert J_{s}\right\vert }$ and to
$0$ otherwise. Define the $m$-linear map $T_{A}:\times_{j\in J_{1}}\ell
_{1}\times\cdots\times\times_{j\in J_{n}}\ell_{1}\rightarrow F$ by
\[
T_{A}(x^{(1)},\ldots,x^{(m)}):=A(u_{1}^{-1}\circ d_{1}(\otimes_{j\in J_{1}%
}x^{\left(  1\right)  }),\ldots,u_{1}^{-1}\circ d_{1}(\otimes_{j\in J_{n}%
}x^{\left(  n\right)  }))
\]
for every $x^{\left(  j\right)  }\in\ell_{1}$.

Since $\mathcal{L}\left(  ^{m}\ell_{1};F\right)  =\mathcal{L}_{\mathcal{I}%
,as\left(  r;1\right)  }\left(  ^{m}\ell_{1};F\right)  $, the following
information completes the proof:
\begin{align*}
T_{A}(x^{(1)},\ldots,x^{(1)},\ldots,x^{(k)},\ldots,x^{(k)})  &  =A(u_{1}%
^{-1}\circ d_{1}(\otimes_{i\in J_{1}}x^{\left(  1\right)  }),\ldots,u_{1}%
^{-1}\circ d_{1}(\otimes_{i\in J_{n}}x^{\left(  n\right)  }))\\
&  =A(u_{1}^{-1}(\otimes_{i\in J_{1}}x^{\left(  1\right)  }),\ldots,u_{1}%
^{-1}(\otimes_{i\in J_{n}}x^{\left(  n\right)  }))=A(x^{(1)},\ldots,x^{(n)}).
\end{align*}

\end{proof}

\begin{theorem}
[Kwapie\'{n}'s Theorem for blocks]Let $\mathcal{I}=\left\{  I_{1},\ldots
,I_{k}\right\}  $ be a family of pairwise disjoint non-void subsets of
$\{1,\ldots,m\}$ such that $\cup_{j=1}^{k}I_{j}=\{1,\ldots,m\}$. If
$n=\min\left\{  \left\vert I_{1}\right\vert ,\ldots,\left\vert I_{k}%
\right\vert \right\}  ,$ then%
\[
\mathcal{L}\left(  ^{m}\ell_{1};\ell_{p}\right)  =\mathcal{L}_{\mathcal{I}%
,as\left(  t;1\right)  }\left(  ^{m}\ell_{1};\ell_{p}\right)  .
\]
with%
\[
t=\left\{
\begin{array}
[c]{c}%
\frac{2p}{np+2p-2},\text{ if }1\leq p\leq2\\
\frac{2p}{np+2},\text{ if }2\leq p\leq\infty.
\end{array}
\right.
\]
Moreover, the parameter $t$ is optimal when $2\leq p\leq\infty$.
\end{theorem}

\begin{proof}
The result follows from a combination of (\ref{sss}) and Lemma \ref{lll}. The
optimality follows from what we have just proved in Section 3.
\end{proof}

When $k=1$ we recover Kwapie\'{n}'s Theorem for absolutely summing multilinear
operators and when $k=m$ we recover Kwapie\'{n}'s Theorem for multiple summing
operators. In the special case $p=2$ we obtain a unified Grothendieck's theorem:

\begin{corollary}
[Unified Grothendieck's Theorem]Let $\mathcal{I}=\left\{  I_{1},\ldots
,I_{k}\right\}  $ be a family of pairwise disjoint non-void subsets of
$\{1,\ldots,m\}$ such that $\cup_{j=1}^{k}I_{j}=\{1,\ldots,m\}$. If
$n=\min\left\{  \left\vert I_{1}\right\vert ,\ldots,\left\vert I_{k}%
\right\vert \right\}  $ then we have%
\[
\mathcal{L}\left(  ^{m}\ell_{1};\ell_{2}\right)  =\mathcal{L}_{\mathcal{I}%
,as\left(  \frac{2}{n+1};1\right)  }\left(  ^{m}\ell_{1};\ell_{2}\right)
\]
and the result is optimal.
\end{corollary}

\section{Other variants of Kwapie\'{n}'s theorem}

In this final section we present partial answers to Problem \ref{p112233}. Of course, using the Inclusion Theorem (Theorem \ref{9990}), provided that
$1\leq u<\frac{2mn}{2mn-1}$, we can prove that $T\left(  A_{1},...,A_{m}\right)
$ is multiple $\left(  t,u\right)  $-summing for a certain $t.$ However, the
following result provides better estimates for other choices of $u:$

\begin{theorem}
Let $T\in\mathcal{L}\left(  ^{m}\ell_{1};\ell_{p}\right)  $ and $A_{k}%
\in\mathcal{L}\left(  ^{n}\ell_{\infty};\ell_{1}\right)  $ for all
$k=1,...,m.$ Then the composition $T\left(  A_{1},...,A_{m}\right)  $ is
multiple $\left(  t,u\right)  $-summing in the following cases:

(i) For $\left(  p,u\right)  \in\lbrack1,2]\times\lbrack\frac{2n}{n+1},2]$ and%
\[
t=\frac{2pu}{4p+2u-pu-4};
\]

(ii) For $\left(  p,u\right)  \in\lbrack2,\infty]\times\lbrack\frac{2n}%
{n+1},2]$ and
\[
t=\frac{2pu}{pu-2u+4};
\]

(iii) For $\left(  p,u\right)  \in\lbrack1,2]\times\lbrack1,\frac{2n}{n+1}]$
and
\[
t=\frac{2np}{2p+np-2};
\]

(iv) For $\left(  p,u\right)  \in\lbrack2,\infty]\times\lbrack1,\frac{2n}%
{n+1}]$ and
\[
t=\frac{2np}{np+2}.
\]

\end{theorem}

\begin{proof}
(i) We proceed as in the proof of \cite[Theorem 2.3]{bayart}.

Note that
\begin{align*}
\widehat{T}\left(  \widehat{A_{1}}\left(  \ell_{u}^{w}(\ell_{\infty})\right)
,...,\widehat{A_{n}}\left(  \ell_{u}^{w}(\ell_{\infty})\right)  \right)   &
\in\ell_{2}\left(  \ell_{1}\right) \\
\widehat{T}\left(  \widehat{A_{1}}\left(  \ell_{u}^{w}(\ell_{\infty})\right)
,...,\widehat{A_{n}}\left(  \ell_{u}^{w}(\ell_{\infty})\right)  \right)   &
\in\ell_{u}\left(  \ell_{2}\right).
\end{align*}
Let us consider $p=1.$ The operators $A_{1},...,A_{n}\in\mathcal{L}\left(
^{n}\ell_{\infty};\ell_{1}\right)  $ are multiple $\left(  2,2\right)
$-summing (see \cite{lama11}) and thus $T(A_{1},...,A_{m})$ is multiple
$\left(  2,2\right)  $-summing when $T\in\mathcal{L}\left(  ^{m}\ell_{1}%
;\ell_{1}\right)  $ and, \textit{a fortiori}, $T(A_{1},...,A_{m})$ is multiple
$\left(  2,u\right)  $-summing.

If $p=2,$ as $\frac{2n}{n+1}\leq u,$ the operators $A_{1},...,A_{n}%
\in\mathcal{L}\left(  ^{n}\ell_{\infty};\ell_{1}\right)  $ are weakly multiple
$\left(  u,u\right)  $-summing (see \cite{lama11}) and, since $u\leq2,$ it is
well-known that $T\in\mathcal{L}\left(  ^{m}\ell_{1};\ell_{2}\right)  $ is
multiple $\left(  u,u\right)  $-summing. Hence $T(A_{1},...,A_{m})$ is
multiple $\left(  u,u\right)  $-summing.

Proceeding as in \cite[Theorem 2.3]{bayart} we have that%
\[
\widehat{T}\left(  \widehat{A_{1}}\left(  \ell_{u}^{w}(\ell_{\infty})\right)
,...,\widehat{A_{n}}\left(  \ell_{u}^{w}(\ell_{\infty})\right)  \right)
\in\ell_{t}\left(  \ell_{p}\right)
\]
for
\[
\frac{1}{t}=\frac{\frac{2-p}{p}}{2}+\frac{1-\frac{2-p}{p}}{_{u}}.
\]

We thus have%

\[
t=\frac{2pu}{4p+2u-pu-4}.
\]

(ii) Note that, as in the first case,
\begin{align*}
\widehat{T}\left(  \widehat{A_{1}}\left(  \ell_{u}^{w}(\ell_{\infty})\right)
,...,\widehat{A_{n}}\left(  \ell_{u}^{w}(\ell_{\infty})\right)  \right)   &
\in\ell_{2}\left(  \ell_{\infty}\right) \\
\widehat{T}\left(  \widehat{A_{1}}\left(  \ell_{u}^{w}(\ell_{\infty})\right)
,...,\widehat{A_{n}}\left(  \ell_{u}^{w}(\ell_{\infty})\right)  \right)   &
\in\ell_{u}\left(  \ell_{2}\right).
\end{align*}

Thus%
\[
\widehat{T}\left(  \widehat{A_{1}}\left(  \ell_{u}^{w}(\ell_{\infty})\right)
,...,\widehat{A_{n}}\left(  \ell_{u}^{w}(\ell_{\infty})\right)  \right)
\in\ell_{t}\left(  \ell_{p}\right)
\]
for%
\[
\frac{1}{t}=\frac{1-\frac{2}{p}}{2}+\frac{\frac{2}{p}}{_{u}},
\]
and thus%

\[
t=\frac{2pu}{pu-2u+4}.
\]

(iii) \ If $p=1,$ $T(A_{1},...,A_{m})$ is multiple $\left(  2,u\right)
$-summing (\cite[Proposition 3.3]{nachr}). If $p=2,$ the operators
$A_{1},...,A_{n}\in\mathcal{L}\left(  ^{n}\ell_{\infty};\ell_{1}\right)  $ are
weakly multiple $\left(  \frac{2n}{n+1},u\right)  $-summing (see
\cite{lama11}) and $T\in\mathcal{L}\left(  ^{m}\ell_{1};\ell_{2}\right)  $ is
multiple $\left(  \frac{2n}{n+1},\frac{2n}{n+1}\right)  $-summing. Thus
$T(A_{1},...,A_{m})$ is multiple $\left(  \frac{2n}{n+1},u\right)  $-summing.

Proceeding as in \cite{bayart}, for $T\in\mathcal{L}\left(  ^{m}\ell_{1}%
;\ell_{p}\right)  $ we have that $T(A_{1},...,A_{m})$ is multiple $\left(
t,u\right)$-summing for%
\[
t=\frac{2np}{2p+np-2}.%
\]

(iv) If $p=\infty,$ the operator $T(A_{1},...,A_{m})$ is multiple $\left(
2,2\right)  $-summing, because $A_{1},...,A_{m}$ are multiple $\left(
2,2\right)  $-summing. (\cite[Proposition 3.3]{nachr}). Thus $T(A_{1}%
,...,A_{m})$ is multiple $\left(  2,u\right)  $-summing

If $p=2,$ as in the previous case we know that $T(A_{1},...,A_{m})$ is
multiple $\left(  \frac{2n}{n+1},u\right)  $-summing.

Proceeding as in \cite{bayart}, for $T\in\mathcal{L}\left(  ^{m}\ell_{1}%
;\ell_{p}\right)  $ we have that $T(A_{1},...,A_{m})$ is multiple $\left(
t,u\right)  $-summing for%
\[
t=\frac{2np}{np+2}.%
\]

\end{proof}

Our final result extends (2) of Theorem \ref{dd22}:

\begin{theorem}
Let $T\in\mathcal{L}\left(  ^{m}\ell_{1};\ell_{p}\right)  $ and $A_{k}%
\in\mathcal{L}\left(  ^{n}\ell_{\infty};\ell_{1}\right)  $ for all
$k=1,...,m.$ Assume that $n\geq2.$ If $s\geq1,$ the composition $T\left(
A_{1},...,A_{m}\right)  $ is absolutely $\left(  r;2s\right)  $-summing for%
\[
r=\left\{
\begin{array}
[c]{c}%
\frac{2ps}{2mp-2s+2ps-mps}\text{ if }1\leq p\leq2\\
\frac{2ps}{2mp+2s-mps}\text{ if }2\leq p\leq\infty.
\end{array}
\right.
\]

\end{theorem}

\begin{proof}
By \cite[Theorem 3.15]{complu} we know that every operator in $\mathcal{L}%
\left(  ^{n}\ell_{\infty};\mathbb{K}\right)  $ is absolutely $\left(
s;2s\right)  $-summing for all $s\geq1$ and hence every $A_{k}$ is weakly
absolutely $\left(  s;2s\right)  $-summing.

Let us suppose $1\leq p\leq2.$ Since every $T\in\mathcal{L}\left(  ^{m}%
\ell_{1};\ell_{p}\right)  $ is absolutely $\left(  \frac{2p}{mp+2p-2};1\right)
$-summing, by the Inclusion Theorem we conclude that every $T\in
\mathcal{L}\left(  ^{m}\ell_{1};\ell_{p}\right)  $ is absolutely $\left(
\frac{2ps}{2mp-2s+2ps-mps};s\right)  $-summing. Thus, $T\left(  A_{1}%
,...,A_{m}\right)  $ is absolutely $\left(  \frac{2ps}{2mp-2s+2ps-mps}%
;2s\right)  $-summing.

If $2\leq p\leq\infty,$ since every $T\in\mathcal{L}\left(  ^{m}\ell_{1}%
;\ell_{p}\right)  $ is absolutely $\left(  \frac{2p}{mp+2};1\right)  $-summing,
by the Inclusion Theorem we conclude that every $T\in\mathcal{L}\left(
^{m}\ell_{1};\ell_{p}\right)  $ is absolutely $\left(  \frac{2ps}%
{2mp-2s+mps};s\right)  $-summing. Thus, $T\left(  A_{1},...,A_{m}\right)  $ is
absolutely $\left(  \frac{2ps}{2mp-2s+mps};2s\right)  $-summing.
\end{proof}


\begin{thebibliography}{99}                                                                                               %


\bibitem {blocos}N. Albuquerque, G. Ara\'{u}jo, W. Cavalcante, T. Nogueira, D.
N\'{u}\~{n}ez-Alarc\'{o}n, D. Pellegrino and P. Rueda, On summability of
multilinear operators and applications, Ann. Funct. Anal., \textbf{9} (2018), 574--590.

\bibitem {nacib}N. Albuquerque and L. Rezende, Anisotropic regularity principle
in sequence spaces and applications, Commun. Contemp. Math., \textbf{20} (2018),
1750087, 14 pp.

\bibitem {AlencarMatos}R. Alencar and M. C. Matos, Some classes of multilinear
mappings between Banach spaces, Publicaciones del Departamento de An\'{a}lisis
Matem\'{a}tico 12, Universidad Complutense Madrid, (1989).

\bibitem {lama11}G. Araujo and D. Pellegrino, Optimal estimates for summing
multilinear operators, Linear and Multilinear Algebra, \textbf{65} (2017), 930--942.


\bibitem {bayart2}F. Bayart, Multiple summing maps: coordinatewise
summability, inclusion theorems and $p$-Sidon sets, J. Funct. Anal., \textbf{274}
(2018), 1129--1154.

\bibitem {bay111}F. Bayart, Summability of the coefficients of a multilinear
form, to appear in J. Eur. Math. Soc.

\bibitem {bayart}F. Bayart, D. Pellegrino and P. Rueda, On coincidence results
for summing multilinear operators: interpolation, $\ell_{1}$-spaces and
cotype, Collect. Math., \textbf{71} (2020), 301--318.

\bibitem {bennett}G. Bennett, Schur multipliers, Duke Math. J., \textbf{44}
(1977), 603--639.

\bibitem {bombal}F. Bombal, D. P\'{e}rez-Garc\'{\i}a and I. Villanueva,
Multilinear extensions of Grothendieck's theorem, Q. J. Math., \textbf{55}
(2004), 441--450.

\bibitem {complu}G. Botelho, C. Michels and D. Pellegrino, Complex
interpolation and summability properties of multilinear operators, Rev. Mat.
Complut., \textbf{23} (2010), 139--161.

\bibitem {nachr}G. Botelho and D. Pellegrino, When every multilinear mapping
is multiple summing, Math. Nachr., \textbf{282} (2009), 1414--1422.

\bibitem {bo1}G. Botelho and D. Freitas, Summing multilinear operators by blocks:
The isotropic and anisotropic cases, J. Math. Anal. Appl., \textbf{490} (2020), 124203, 21 pp.



\bibitem {defant}A. Defant, D. Popa and U. Schwarting, Coordenatewise multiple
summing operators on Banach spaces, J. Funct. Anal., \textbf{259} (2010), 220--242.

\bibitem {diestel}J. Diestel, H. Jarchow and A. Tonge, Absolutely summing
operators. Cambridge Studies in Advanced Mathematics, 43. Cambridge University
Press, Cambridge, 1995.

\bibitem {K}S. Kwapie\'{n}, Some remarks on $(p,q)$-absolutely summing
operators in $\ell_{p}$-spaces, Studia Math., \textbf{29} (1968), 327--337.


\bibitem {collec}M.C. Matos, Fully absolutely summing and Hilbert-Schmidt
multilinear mappings, Collect. Math., \textbf{54} (2003), 111--136.

\bibitem {ps}D.\ Pellegrino and J. Seoane-Sep\'{u}lveda, Grothendieck's theorem
for absolutely summing multilinear operators is optimal, Linear Multilinear
Algebra, \textbf{63} (2015), 554--558.

\bibitem {PSDianaE}D. Pellegrino, J. Santos, D. Serrano-Rodr\'{\i}guez and E.
Teixeira, A regularity principle in sequence spaces and applications, Bull.
Sci. Math., \textbf{141} (2017), 802--837.


\bibitem {jmaa}D. P\'{e}rez-Garc\'{\i}a and I. Villanueva, Multiple summing
operators on Banach spaces, J. Math. Anal. Appl., \textbf{285} (2003), 86--96.

\bibitem {Ry}R. Ryan, Introduction to tensor products, Springer Verlag,
London, 2002.
\end{thebibliography}
\end{document}